\documentclass{amsart}

\usepackage{amsmath,amssymb,amsthm}
\usepackage[dvipdfm]{graphicx}
\usepackage{enumerate}
\usepackage[dvips]{color}
\usepackage{version}

\newtheorem{Thm}{Theorem}[section]

\newtheorem{theorem}[Thm]{Theorem}
\newtheorem{lemma}[Thm]{Lemma}

\newtheorem{corollary}[Thm]{Corollary}
\newtheorem{Prop}[Thm]{Proposition}

\theoremstyle{remark}

\newtheorem{Ex}[Thm]{Example}
\theoremstyle{definition}
\newtheorem{Def}[Thm]{Definition}
\newtheorem{Def-Pro}[Thm]{Definition-Proposition}

\newcommand{\Supp}{\mathop{\mathrm{Supp}}\nolimits}

\excludeversion{Eliminates}

\excludeversion{N.B.}

\excludeversion{memo}

\begin{document}

\title[semi-log-canonical models]
{ log-canonical models of singular pairs and its applications}

\author{Yuji Odaka}
\address{Research Institute for Mathematical Sciences, 
Kyoto University, Japan} 
\email{yodaka@kurims.kyoto-u.ac.jp}
\urladdr{\texttt{http://www.kurims.kyoto-u.ac.jp/\~{}yodaka}}

\author{Chenyang Xu}
\address{Beijing International Center of Mathematics Research, 5 Yiheyuan Road, Haidian District, Beijing 100871, China}
\address{Department of Mathematics\\
University of Utah\\
155 South 1400 East\\
Salt Lake City, UT 84112, USA}
\email{dr.chenyang.xu@gmail.com}
\date{August 10th, 2011. Revised: November 14th, 2011.}
\maketitle

\begin{abstract}
We prove the existence of the log canonical model over a log pair $(X,\Delta)$.
As an application, together with Koll\'ar's gluing theory, we remove the assumption in the first named author's work \cite{Odaka11}, which shows that $K$-semistable
polarized varieties can only have semi-log-canonical singularities.
\end{abstract}
\tableofcontents

\section{Introduction}
Throughout this paper, the ground field is assumed to be an algebraically closed field
of characteristic 0.
It is well known that a  normal surface singularity has the \textit{minimal resolution},
while for a singular variety of higher dimension, usually it does not have any such  ``canonically determined" smooth modification. But if we allow the partial resolution having mild singularities, a type of singularities coming from the minimal model program (MMP) which is natural for many questions,
then it is possible. More precisely, for an arbitrary normal variety $X$,
we can consider a unique (``canonically determined") partial resolution $Y\to X$ with only canonical singularities and satisfies the property that $K_Y$ is relative ample over $X$. The existence of such model $Y$, i.e., \textit{the canonical model over $X$}
\footnote{Here, the adjective ``canonical" comes from the sense of singularities. },
is implied by \cite[Main theorem (1.2)]{BCHM10}. In the case of surfaces, $Y$ is obtained by  contracting all exceptional curves with self-intersection $(-2)$ from the minimal resolution.

Similarly, for a normal pair $(X,\Delta)$ i.e., attached with a boundary $\mathbb{Q}$-divisor, we can define a ``canonically determined" partial resolution
$(Y,\Delta_Y)\to (X,\Delta)$ associated to it, which is called its \textit{log canonical model} (see \eqref{d-lc}).
It coincides with the relative log canonical model of a log resolution with
a reduced  boundary,
in the sense of the usual relative log MMP, as we will show in Lemma \ref{l-mmp}.

In this note, we study the question of the existence of  log canonical model of a normal  pair $(X,\Delta)$. It is well-known that the full log MMP (including the abundance conjecture) gives an affirmative answer to the question. As the full log MMP is still not established, our main observation in this note is that if we assume $K_X+\Delta$ is $\mathbb{Q}$-Cartier, then the existence of log canonical model follows from the established results on MMP, especially the recent ones in \cite{Birkar11} and \cite{HX11}.

\begin{theorem}\label{t-lcm}
 Let $(X,\Delta)$ be a normal pair,
i.e., $X$ is a normal variety and $\Delta=\sum a_i\Delta_i$ is a $\mathbb{Q}$-divisor
with distinct prime divisors $\Delta_i$ and rational numbers $a_i$.
Assume $0\le a_i \le 1$ and $K_X+\Delta$ is $\mathbb{Q}$-Cartier. Then there exists a \textit{log canonical model} $(Y,\Delta_Y)$ over $(X,\Delta)$ (see \eqref{d-lc} for the definition).
\end{theorem}

As a consequence, we give a proof of the inversion of adjunction for log canonicity, which is a slight simplification of Hacon's argument in \cite{Hacon11} (also see \cite[4.11.2]{Kol}). 
We note that the inversion of adjunction for log canonicity was first proved by Kawakita (cf. \cite{Kawakita07}) without using the minimal model program.
\begin{corollary}[Inversion of Adjunction]\label{c-ia}
 Let $(X,D+\Delta)$ be a normal pair and $D$ a reduced divisor. Assume
$K_X+D+\Delta$ is $\mathbb{Q}$-Cartier. Let $n\colon D^n\to D$ be the normalization, and write $n^*(K_X+D+\Delta)|_{D}=K_{D^n}+\Delta_{D^n}$.

Then $(X,D+\Delta)$ is log canonical along $D$ if and only if $(D^n, \Delta_{D^n})$ is log canonical.
\end{corollary}

We can also extend our results into non-normal setting.
In fact, Koll\'ar recently has developed a rather complete theory of  semi-log-canonical pairs by studying their normalizations. Thanks to his fundamental theory (see \cite{Kol}), including his recent result \cite{Kol11}, we have the following as a consequence, which generalizes Theorem
\ref{t-lcm}.

\begin{corollary}\label{t-slcm}
 Let $(X,\Delta)$ be a demi-normal pair where $\Delta=\sum a_i\Delta_i$ is a $\mathbb{Q}$-divisor, none of prime divisor $\Delta_i$ are in the singular locus ${\rm Sing}(X)$. Assume $0\le a_i \le 1$ and $K_X+\Delta$ is $\mathbb{Q}$-Cartier. Then the semi-log-canonical
model $(Y,\Delta_Y)$ over $(X,\Delta)$ exists.
 \end{corollary}

Recall that \textit{demi-normality} of $X$ means that it is normal crossing
in codimension $1$ and satisfies Serre's $S_2$ condition \cite[5.1]{Kol}.
For the precise definition of \textit{semi-log-canonical model}, see
Definition \ref{slc-mod}.

One of our main applications for this note is the following:
In \cite{Odaka11} the first named author proved $K$-semi-stability implies semi-log canonicity, assuming  the existence of semi-log-canonical models.   Since  \eqref{t-slcm} verifies this  assumption, the following theorem now becomes unconditional.
\begin{theorem}[\cite{Odaka11}] Let $X$ be an equidimensional reduced projective variety, satisfies $S_2$ condition and whose codimension 1 points are Gorenstein. Thus we can define the Weil divisor class $K_X$ which we assume to be $\mathbb{Q}$-Cartier. 

Then, if  $(X,L)$ is  $K$-semistable, $X$ has only semi-log-canonical singularities.
\end{theorem}

Roughly speaking, assuming the non-semi-log-canonicity of $X$,
\cite{Odaka11} proved that we can construct ``destabilizing test configuration"
by using the semi-log-canonical model of $X$.
We refer to \cite{Odaka11} for more details. 

\section{Log canonical models}\label{pr}

\begin{Def}\label{d-lc}
 Let $(X,\Delta)$ be a normal pair,
i.e., $X$ is a normal variety and $\Delta=\sum a_i\Delta_i$ is a $\mathbb{Q}$-divisor
with distinct prime divisors $\Delta$ and rational numbers $a_i$.
Assume $0\le a_i \le 1$.
We call that a birational projective morphism $f: Y\to (X,\Delta)$  give  {\it a log canonical model} over $(X,\Delta)$  if with  the divisor $\Delta_Y=f^{-1}_*(\Delta_X)+E_f^{lc}$ on $Y$, where $E_f^{lc}$ denotes the sum of $f$-exceptional prime divisors with coefficients $1$, the pair $(Y,\Delta_Y)$ satisfies
\begin{enumerate}
\item[(1)] $(Y,\Delta_Y)$ is log canonical,
\item[(2)] $K_Y+\Delta_Y$ is ample over $X$.
\end{enumerate}

\end{Def}

From the negativity lemma (cf. \cite[3.38]{KM98}), we know that $f: Y \to X$ is isomorphic over the maximal open locus $X^{lc}$  on which $(X,\Delta)$ is log canonical (see the proof of \eqref{l-con}).
For more background of log canonical models over a pair $(X,\Delta)$, see \cite[Section 2]{Koletc92}. 

First, we discuss the uniqueness of the log-canonical model. 

\begin{lemma}\label{l-mmp}
Let $\tilde{f}: \tilde{Y}\to X$ be a log resolution of $(X,\Delta)$.
Assume that $(\tilde{Y},\Delta_{\tilde{Y}}:=\tilde{f}^{-1}_*\Delta+\sum E_i)$ has a relative log canonical model $(Y,\Delta_Y)$ over $X$, where $E_i$ run over all $\tilde{f}$-exceptional prime divisors. Then $Y\to (X,\Delta)$ is a log canonical model over $(X,\Delta)$.
\end{lemma}
\begin{proof}By the definition of the relative log canonical model, $(Y,\Delta_Y)$ obviously satisfies the conditions (1) and (2).
\end{proof}

\begin{Prop}\label{P-uni}
If log canonical model $Y$ exists, then it is unique. 
\end{Prop}
\begin{proof}Let $g:\tilde{Y}\to Y$ be a log resolution of $(Y,f^{-1}_*(\Delta)+{\rm Ex}(f))$. And we write 
$$g^*(K_Y+\Delta_Y)+E\sim_{\mathbb{Q}}K_{\tilde{Y}}+F,$$
such that $E,\ F\ge 0$, have no common components. It is easy to see that $g_*(F)=\Delta$.  
 Since $(Y,\Delta_Y)$ is log canonical $$\tilde{f}^{-1}_*\Delta+\sum E_i\ge F,$$ where $\tilde{f}=f\circ g$ and $E_i$ run over all $\tilde{f}$-exceptional prime divisors. The difference is $g$-exceptional. We conclude that
$${\rm Proj} \bigoplus _{m\in \mathbb{Z}_{\geq0}} \tilde{f}_*\mathcal{O}_{\tilde{Y}}(m(K_{\tilde{Y}}+\Delta_{\tilde{Y}}))\cong{\rm Proj} \bigoplus _{m\in \mathbb{Z}_{\geq0}} f_*\mathcal{O}_{Y}(m(K_{Y}+\Delta_{Y}))\cong Y,$$
as $K_Y+\Delta_Y$ is ample over $X$. So it suffices to show that the different log resolutions as in \eqref{l-mmp} will yield the same log canonical model $Y$.

We assume that there are two difference choices $g_i\colon (\tilde{Y}_i,\Delta_{\tilde{Y_i}}) \to (X,\Delta)$ $(i=1,2)$ with a morphism $\mu:\tilde{Y}_1\to \tilde{Y}_2$. Since $\mu^*(K_{\tilde{Y}_2}+\Delta_{\tilde{Y_2}})+E'=K_{\tilde{Y}_1}+\Delta_{\tilde{Y_1}}$ for some effective exceptional divisor $E'$. The uniqueness immediately follows from the fact that $(\tilde{Y}_i,\Delta_{\tilde{Y_i}})$ have the same relative log canonical ring
(sheaf)$$\bigoplus _{m\in \mathbb{Z}_{\geq0}} (g_i)_*\mathcal{O}_{\tilde{Y_i}}(m(K_{\tilde{Y_i}}+\Delta_{\tilde{Y_i}}))$$ over $X$.
\end{proof}

\begin{lemma}\label{l-con}Let $(X,\Delta)$ be a pair as in \eqref{d-lc}. We assume that  $K_X+\Delta$ is $\mathbb{Q}$-Cartier.
Let $f:Y\to (X,\Delta)$ be the log canonical model. Write
$${f}^*(K_X+\Delta)\sim_{\mathbb{Q}}K_{Y}+B,$$
and $B=\sum b_iB_i$  as the sum of distinct prime divisors such that $f_*(B)=\Delta$, we let  $B^{>1}$ be the nonzero divisor $\sum_{b_i>1}b_iB_i$ and $B^{\le1}$ be the divisor $\sum_{b_i\le 1}b_iB_i$,
then ${\rm Supp}(B^{>1})={\rm Ex}(f)$. In particular, ${\rm Ex}(f)\subset Y$ is of pure codimension 1.
\end{lemma}
\begin{proof} It is obvious that ${\rm Supp}(B^{>1})\subset{\rm Ex}(f)$.

If we write $B=f_*^{-1}(\Delta)+E_B$, then $E_B$ is supported on the exceptional locus and the divisor $E_f^{lc}-E_B$ is an exceptional divisor which is relatively ample.  It follows from the negativity lemma (cf. \cite[3.38]{KM98}) that $E_f^{lc}-E_B\le 0$. Therefore, we have the equality
$$f_*^{-1}(\Delta)+E_f^{lc}=B^{\le 1}+{\rm Supp}(B^{>1}).$$

From the definition of the log canonical model (\ref{d-lc}), we know that 
\begin{equation}\label{e-eq}
K_{Y}+{\rm Supp}(B^{>1})+B^{\le 1}\sim_{\mathbb{Q},X} {\rm Supp}(B^{>1})-B^{>1}=\sum_{b_i>1}(1-b_i)B_i
\end{equation} is relatively ample over $X$.
Thus for any curve $C$ which is contracted by $f$, we have
$$C\cdot  (\sum_{b_i>1}(b_i-1)B_i)<0,$$
which implies that $C\subset {\rm Supp}(B^{>1})$.
This shows ${\rm Ex}(f)\subset \Supp(B^{>1})$ which completes the proof.
\end{proof}

\begin{proof}[Proof of \eqref{t-lcm}]
We take a ($\mathbb{Q}$-factorial) dlt modification $g:Z\to X$ of $(X,\Delta)$ (cf.\ \cite[Section 3]{KK10}or \cite[4.1]{Fujino10})
such that
\begin{enumerate}
\item if we write $g^*(K_X+\Delta)\sim_{\mathbb{Q}}K_Z+g_*^{-1}(\Delta)+\sum b_iE_i$, then $b_i\ge 1$;
\item $(Z,\Delta_Z=g_*^{-1}(\Delta)+\sum E_i )$ is dlt.
\end{enumerate}

We remark $Z$ can be achieved by running a sequence of $(K_{\tilde{Y}}+\Delta_{\tilde{Y}})$-MMP over $X$ for a log resolution of $\tilde{f}:\tilde{Y}\to (X,\Delta)$ (cf. \cite[4.1]{Fujino10}) where $\Delta_{\tilde{Y}}$ is defined as in \eqref{l-mmp}. Furthermore, we require that $\tilde{f}^{-1}(X\setminus X^{lc})$ is a divisor. We want to show $(Z,\Delta_Z=g_*^{-1}(\Delta)+\sum E_i )$ has a good minimal model over $X$. Since then we can take $Y$ to be the relative log canonical model, which is easy to see it is the log canonical model of $(X,\Delta)$.

\begin{lemma}\label{l-dlt}
Let $V$ be a log canonical center of $(Z,{g}_*^{-1}(\Delta)+\sum E_i)$, such that $g(V)\subset X\setminus X^{lc}$, then $V\subset E_i $ for some $i$.
\end{lemma}
\begin{proof}As $Z$ is obtained by running a sequence of MMP for a log smooth resolution $\tilde{f}: \tilde{Y}\to (X,\Delta)$, then $V$ is an lc center of  $(Z,{g}_*^{-1}(\Delta)+\sum E_i)$ if and only if $\tilde{Y}\dashrightarrow Z$ is isomorphic over the generic point $V$ and the preimage $W$ of $V$ in $\tilde{Y}$ is a component of $\cap F_i$, where $F_i$'s are prime divisors contained in $\lfloor \Delta_{\tilde{Y}}\rfloor$.
As by our assumption $\tilde{f}^{-1}(X\setminus X^{lc})$ is a union of divisors, if $\tilde{f}(W)\subset X\setminus X^{lc}$, then $W$ is contained in one of the $\tilde{f}$-exceptional divisors $\tilde{E}_i$ whose image is in $X\setminus X^{lc}$. Therefore. $V$ is contained the birational transform of $\tilde{E}_i$ on $Z$ as $\tilde{Y}\dashrightarrow Z$ is an isomorphism on the generic point of $\tilde{E}_i$.
\end{proof}

Now consider all exceptional divisors $E$ of $g$ with the centers contained in $X\setminus X^{lc}$. Fixing a general relatively ample effective divisor $H$ on $Z$ over $X$, we run $(K_Z+\Delta_Z)$-MMP with scaling of $H$ over $X$ (cf., \cite[subsection 3.10]{BCHM10}). As we treat dlt pairs, which are not klt, we explain what follows
from \cite{BCHM10} in the following Lemma for the readers' convenience.

\begin{lemma}
We can run the MMP with scaling of $H$ for $(Z,\Delta_Z)$ over $X$
to get a sequence of numbers $0\le...\le s_{2}\le s_1\le s_0$ and a sequence of birational models $$Z=Z_0\dasharrow Z_1\dasharrow Z_2\dasharrow\cdots, $$ such that
the following holds. Here, $\Delta_j$ and $H_j$ are push-forwards of $\Delta$ and $H$
on each $Z_j$.

{\rm (i)}$K_{Z_j}+\Delta_j+tH_j$ is semi-ample over $X$ for any $s_j\ge t \ge s_{j+1}$.

{\rm (ii)}This sequence $\{s_i\}$ $($is either finite with $\exists s_N=0$
or$)$ satisfies the property that $\lim_j s_j=0$.

\end{lemma}
\begin{proof}
For each $Z_j$, we set $$s_{j+1}:=\inf \{t>0 \mid K_{Z_{j}}+\Delta_{j}+tH_{j} \mbox{ is relatively nef over }X \}$$ and consider the extremal contraction of an extremal ray $R_j$ with $(K_{Z_j}+\Delta_{j}+s_{j+1}H_{j})\cdot R_j=0$.
In each step the existence of flip  holds
since as $K_{Z_j}+\Delta_j$ is dlt and $Z_{j}$ is $\mathbb{Q}$-factorial, $(K_{Z_j}+\Delta_j)$-flip is the same as a step of $(K_{Z_j}+(1-\delta)\Delta_j)$-MMP for $0<\delta\ll 1$ and $({Z_j}, (1-\delta)\Delta_j)$ is klt (cf., \cite[Corollary 1.4.1]{BCHM10}).

From our construction, we know that giving a sequence of $j$ steps
 $$Z=Z_0\dasharrow Z_1\dasharrow Z_2\dasharrow Z_j $$
of $(K_Z+\Delta)$-MMP with scaling of $H$ as above is the same as giving a sequence of steps of $(K_Z+\Delta+tH)$-MMP with scaling of $H$ for any $0\le t<s_j$.

 For arbitrary $t>0$, there exists an effective divisor $\Theta_t\sim_{\mathbb{Q}}\Delta_Z+tH$, such that $(Z,\Theta_t)$ is klt
(with $\Theta_t$ is relatively big, which is trivial in this case
since $Z$ is birational over $X$).
It follows from \cite[Corollary 1.4.2, see also Theorem 1.2]{BCHM10}
 that  any sequence of $(K_Z+\Theta_t)$-MMP with scaling of $H$ over $X$ will terminate after finite steps  with a relative good minimal model $Z_j$, i.e., $K_{Z_j}+{\rho_j}_*(\Theta_t)$ is semi-ample over $X$ where $\rho_j: Z\dasharrow Z_j$ is the birational contraction. (Recall that \textit{good minimal model} means a
minimal model which satisfies the abundance conjecture.) Thus, {\rm (i)} is proved.

Moreover, from the arguments above, there are only finitely many $s_j$ such that $s_j> t$. Since we can choose $t$ to be an arbitrarily small positive number, we also have the conclusion {\rm (ii)}.
\end{proof}

The {\it diminished stable base locus}
\footnote{also called {\it restricted stable base locus}. }
of $(Z,\Delta_Z)$ over $Z$ is defined by
$${\bf B_{-}}(K_Z+\Delta_Z/X)=\bigcup _{\epsilon>0}{\bf B}(K_Z+\Delta_Z+\epsilon H/X),  $$where ${\bf B}(\cdot)$ denotes the usual stable base locus.
If there is a divisor $E\subset {\bf B_{-}}(K_Z+\Delta_Z/X) $, then $E\subset {\bf B}(K_Z+\Delta_Z+tH/X)$ for some $t>0$, therefore there exists an $j$, such that $s_j \ge t \ge s_{j+1 }$.
Since
$$K_{Z_j}+ \Delta_j+tH_j\sim_{\mathbb{Q}_j}K_{Z_j}+{\rho_j}_*\Theta_t$$ is semiample over $X$
we know that $\rho_j$ contracts $E$.
\begin{lemma}
There exists $Z_j$ such that if we denote by $Z'=Z_j$, $\rho'=\rho_j$, the morphism $g':Z'\to X$ and write
$${g'}^*(K_X+\Delta)=K_{Z'}+{g'}_*^{-1}(\Delta)+\sum b_{i}E'_i,$$ then $b_i>1$ for all $E'_i$ which centers in $X\setminus X^{lc}$. 
\end{lemma}
\begin{proof}
From the above discussion, we can assume that there is $Z_j=Z'$ such that ${\bf B}_{-}(K_{Z'}+{g'}_*^{-1}(\Delta)+\sum E'_i)$ has codimension at least 2. By \eqref{e-eq}, we have
$$K_{Z'}+\rho'_*(\Delta_Z)=K_{Z'}+{g'}_*^{-1}(\Delta)+\sum E'_i\sim_{\mathbb{Q},X}\sum(1-b_i)E'_i,$$  if the statement is not true,
it follows from the Koll\'ar-Shokurov's Connectedness Theorem (cf. \cite[17.4]{Koletc92}) that there is a divisor $E_0'$, with $b_0=1$ such that $\sum_{b_i>1}E'_i|_{E_0'}$ is not trivial.
Therefore, $$(\sum(1-b_i)E'_i+\epsilon H')|_{E_0'}$$ is not effective for small $0<\epsilon \ll 1$, where $H':=\rho_*H$. This implies that $E_0'\subset  {\bf B_{-}}(K_Z+\Delta_Z/X)$, which yields a contradiction.
Then we conclude that $b_i>1$ for all $E'_i$ whose center is in $X\setminus X^{lc}$.
\end{proof}

Now consider the dlt pair $(Z',{g'}_*^{-1}(\Delta)+\Sigma)$, where $\Sigma=\sum E'_i-\epsilon\sum(b_i-1)E'_i$ for some positive $\epsilon \ll 1$.

\begin{lemma} $(Z',{g'}_*^{-1}(\Delta)+\Sigma)$  has a good minimal model $Y'$ over $X$.
\end{lemma}
\begin{proof} Over the open set  $X^{lc}$, we have
 $$(K_{Z'}+{g'}_*^{-1}(\Delta)+\Sigma)|_{{g'}^{-1}(X^{lc})}={g'}^*(K_X+\Delta)|_{{g'}^{-1}(X^{lc})},$$
 whose ring of pluri-log canonical sections  is finitely generated over $X^{lc}$, because it is isomorphic to the algebra
 $$\oplus_{m\ge 0}\mathcal{O}_{X^{lc}}(m(K_{X^{lc}}+\Delta|_{X^{lc}})).$$ Therefore, the restriction of $(Z',{g'}_*^{-1}(\Delta)+\Sigma))$ over $X^{lc}$  has a relative good minimal model over $X^{lc}$ by \cite[2.11]{HX11}. Any lc center  of $(Z',{g'}_*^{-1}(\Delta)+\sum E_i')$ which is contained  in on of $E_i'$ can not be an lc center of $(Z',{g'}_*^{-1}(\Delta)+\Sigma)$, however these lc centers are precisely those centers 
 of $(Z',{g'}_*^{-1}(\Delta)+\sum E_i')$ which is mapped into $X\setminus X^{lc}$ by \eqref{l-dlt}. Thus we conclude that if $V$ is an lc center of $(Z',{g'}_*^{-1}(\Delta)+\Sigma)$, 
 then its image under $g'$ intersects $X^{lc}$. Therefore, it follows from \cite[Theorem 1.9]{Birkar11} or \cite[1.1]{HX11} that $(Z',{g'}_*^{-1}(\Delta)+\Sigma)$  has a good minimal model $f'\colon Y'\to X$.
\end{proof}

Since
$$K_{Z'}+{g'}_*^{-1}(\Delta)+\sum E'_i=\frac{1}{1+\epsilon}(K_{Z'}+{g'}_*^{-1}(\Delta)+\Sigma)+\frac{\epsilon}{1+\epsilon}{g'}^*(K_X+\Delta),$$ we conclude that
$Y'$ is also a relative good minimal model for $K_{Z'}+{g'}_*^{-1}(\Delta)+\sum E'_i$ over $X$. 
\end{proof}

\begin{proof}[Proof of \eqref{c-ia}]One direction is easy (cf. \cite[17.2]{Koletc92}). To prove the converse, let us assume that $(X,D+\Delta)$ is not log canonical along $D$.
Let $f:Y\to (X,D+\Delta)$ be the log canonical model
as in the proof of \eqref{t-lcm}. Write
$${f}^*(K_X+D+\Delta)=K_{Y}+D_{Y}+B,$$ where $D_{Y}$ is the birational transform of $D$.  Since $f$ is not an isomorphism over $D$, it follows from \eqref{l-con} that
$$D_{Y}\cap {\rm Ex}(f) =D_Y \cap {\rm Supp}(B^{>1})\neq\emptyset.$$ Therefore, if we denote by $D_Y^n$ the normalization of $D_Y$ and write
$$n^*(K_{Y}+D_{Y}+B)|_{D_{Y}}=K_{D^n_{Y}}+B_{D^n_{Y}},$$ then $(D^n_{Y},B_{D^n_{Y}})$ has coefficient strictly larger than 1 along some components of $D_Y\cap {\rm Ex }(f)$ by \eqref{l-con},  which implies that $(D^n, \Delta_{D^n})$ is not log canonical.
\end{proof}

\begin{corollary}\label{c-adj}
Notation as above proof. Let $f\colon Y\to (X,D+\Delta)$ be the log canonical model.
Let $D_Y$ be the birational transformation of $D$ and $n\colon D_Y^n\to D_Y$ its normalization. 
Then $f_{D^n}\colon D^n_Y\to (D^n,\Delta_{D^n})$ is also
the log canonical model.
\end{corollary}
\begin{proof}
From the proof of \eqref{c-ia}, we know that
$$n^{-1}({\rm Ex}(f))={\rm Ex}(f_{D^n}),$$
which implies that if we denote ${f_{D^n}^{-1}}_*(\Delta_{D^n})+E^{lc}_{f_{D^n}}$ by $\Delta_{D^n_Y}$, then
$$K_{D^n_Y}+\Delta_{D^n_Y}=n^*((K_{Y}+\Delta_Y)|_{D_Y}) .$$

Then obviously $(D^n_Y,\Delta_{D^n_Y})$ is log canonical and $K_{D^n_Y}+\Delta_{D^n_Y}$ is ample over $D^n$.
\end{proof}

\section{Semi-log-canonical models}
In this section, we study the existence of semi-log canonical model of a demi-normal pair $(X,\Delta)$.  A pair $(X,\Delta)$ is called {\it demi-normal} if $X$ is $S_2$, whose codimension 1 points are regular or ordinary nodes and $\Delta$ is an effective $\mathbb{Q}$-divisor whose support does not contain any  codimensional 1 singular points. For such a demi-normal scheme $X$, let $n:\bar{X}\to X$ be its normalization, we can define the {\it conductor ideal}
$${\rm cond}_X :={\rm Hom}_X(n_* \mathcal{O}_{\bar{X}}, \mathcal{O}_X)\subset \mathcal{O}_X.$$
and the {\it conductor scheme} $D:= {\rm Spec}_X(\mathcal{O}_X/{\rm cond}_X). $
Let $n:\bar{X}\to X$ be the normalization, and $\bar{D}$ the pre-image of $D$ in $\bar{X}$. Then there is an involution $\sigma: \bar{D}^n\to \bar{D}^n$ on the normalization of $\bar{D}$.
We can write $$n^*(K_X+\Delta)\sim_{\mathbb{Q}}K_{\bar{X}}+\bar{D}+\bar{\Delta},$$
where $\bar{\Delta}$ is the preimage of $\Delta$. In fact, we only need to check this formula at all codimension 1 points, which is straightforward.

\begin{Def}We call a demi-normal pair $(X,\Delta)$ is \textit{semi-log-canonical} if $K_X+\Delta$ is $\mathbb{Q}$-Cartier and in the above notations, the pair $(\bar{X},\bar{D}+\bar{\Delta})$ is log canonical. 
\end{Def}

\begin{Def}\label{slc-mod}
Let $(X,\Delta)$ be a demi-normal pair where $\Delta= \sum a_i\Delta_i$ is a sum of distinct prime divisors, none of which is contained in the singular locus ${\rm Sing} (X)$ of $X$, and assume $0 \le a_i \le 1$ for every $i$.

We call a biratonal projective morphism
 $f: Y\to (X,\Delta)$ a {\it semi-log-canonical model}
if $f$ is isomorphic over open locus of $X$ with complement's codimension greater than $1$, and $(Y,\Delta_Y)$ is semi-log-canonical for $\Delta_Y=f_*^{-1}\Delta+E_f^{lc}$ where $E_f^{lc}$ is the sum of all the exceptional prime divisors, and $K_{Y}+\Delta_Y$ is $f$-ample. \end{Def}

We note that from the definition, the induced map on the conductor schemes $D_Y\to D$ is an isomorphism outside some lower dimensional subsets of $D_Y$ and $D$, i.e.,  the codimension 1 points of the $f$-exceptional locus are all regular.
\begin{lemma}Given a demi-normal pair $(X,\Delta)$, its semi-log-canonical model, if exists, is unique.
\end{lemma}
\begin{proof}
Let $Y$ be a semi-log-canonical model of $(X,\Delta)$ and $n_Y:\bar{Y}\to Y$ its normalization and $\bar{f}:\bar{Y}\to \bar{X}$ the induced morphism. We write
$$n_Y^*( K_Y+\Delta_Y)=K_{\bar{Y}}+\bar{D}_Y+\bar{\Delta}_Y.$$
Then $\bar{D}_Y+\Delta_Y=\bar{f}^{-1}_*(\bar{D}+\bar{ \Delta})+E^{lc}_{\bar{f}}$.
Therefore,  $\bar{f}\colon \bar{Y}\to (\bar{X}, \bar{D}+\bar{\Delta}  )$ is the log canonical model, which is unique by \eqref{P-uni}. On a dense open subset of $\bar{D}_Y$, the involution $\sigma_Y: \bar{D}_Y \to \bar{D}_Y$ is the same as the restriction of $\sigma: \bar{D}\to \bar{D}$ to an isomorphic open subset, so $\sigma_Y$ is uniquely determined, hence the quotient $Y$ uniquely exists by \cite[5.3]{Kol}.
\end{proof}

On the other hand, with the results in \cite{Kol11} (also see \cite[Section 4]{Kol}), we can glue the log canonical model of each component of the normalization $\bar{X}\to X$ to get the semi-log canonical model of $(X,\Delta)$.

\begin{proof}[Proof of \eqref{t-slcm}]  Let $f:\bar{Y}\to (\bar{X}, \bar{D}+\bar{\Delta}) $ be the log canonical model and write
$$f^*(K_{\bar{X}}+\bar{D}+\bar{\Delta})=K_{\bar{Y}}+\bar{D}_Y+\bar{\Delta}_Y,$$
where $\bar{D}_Y$ is the birational transform of $\bar{D}$ on $\bar{Y}$. Then it follows from \eqref{c-adj} that the normalization $D_Y^n$ of $\bar{D}_Y$ is the log canonical model of  $(D^n,  \Delta_{D^n})$, where $K_{D^n}+\Delta_{D^n}=n^*(K_{\bar{X}}+\bar{D}+\bar{\Delta})|_{D}$ if we denote
the normalization as $n\colon D^n\to D$.

 Furthermore, because of the uniqueness of the log canonical model by \eqref{P-uni}, this involution $\sigma\colon D^n\to D^n$ can be lifted to an involution on the log canonical model as $\sigma_Y\colon D^n_Y\to D^n_Y$. Since $K_{\bar{Y}}+\bar{D}_Y+\bar{\Delta}_Y$ is ample over $X$, by \cite[26]{Kol11}, $(\bar{Y},\bar{D}_Y,\bar{\Delta},\sigma_Y)$ has a quotient $Y$ which is easy to see to be the semi-log-canonical model of $(X,\Delta)$.
\end{proof}

 While the log canonical models \eqref{t-lcm} are expected to exist even without the assumption that $K_X+\Delta$ is $\mathbb{Q}$-Cartier (cf. \eqref{l-mmp}), the next example constructed by Professor J.~Koll\'ar shows that in (\ref{t-slcm}) the $\mathbb{Q}$-Cartier assumption on $K_X+\Delta$ is necessary.
We are grateful to him for providing this example to us.

\begin{Ex}[Koll\'ar's example on non-existence of semi-log-canonical models]\label{ex}
We construct a demi-normal threefold $X$ with two
irreducible components  $(X_i, D_i)$ such that
$X$ does not have an semi-log-canonical model.
Take
$$
(X_1, D_1):=\bigl(\mathbb{A}^3_{uvw}/\tfrac13(1,1,1), (w=0)/\tfrac13(1,1)
\cong \mathbb{A}^2_{uv}/\tfrac13(1,1)\bigr).
$$
$(X_1, D_1)$ is lc (even plt), hence its log canonical model is trivial, i.e.,
$\pi_1:(Y_1, D_{Y_1})\cong (X_1, D_1)$.

Note that $ \mathbb{A}^2_{uv}/\tfrac13(1,1)$ embeds in $\mathbb{A}^4_{xyzt}$
as the cone over the twisted cubic by
$\sigma: (u,v)\mapsto (u^3, u^2v, uv^2, v^3)$; let $D_2\subset \mathbb{A}^4_{xyzt}$
be its image. Then set
$$
(X_2, D_2):=\bigl((xt-yz=0), D_2\bigr)\subset \mathbb{A}^4_{xyzt}.
$$
Use $\sigma: D_1\cong D_2$ to glue $(X_1, D_1)$ and $(X_2, D_2)$
to obtain $X$.

To compute the log canonical model over $(X_2, D_2)$, note that
$D_2$ satisfies the equation  $xz=y^2$ and
$X_2\cap (xz=y^2)$ is the union of $D_2$ and a residual plane $P:=(x=y=0)$.

Let $\pi_2:Y_2\to X_2$ be the blow up of the  plane $P$
and $C\subset Y_2$ the exceptional curve.
Let $D_{Y_2}$ (resp., $\tilde{P}$) denote the birational transforms $D_2$ (resp., $P$).
Then $\pi_2^*(D_2+P)=D_{Y_2}+\tilde{P}$ and $(C\cdot \tilde{P})=\mathcal{O}(-1)|_{\mathbb{P}^1}=-1$. Thus
$(C\cdot D_{Y_2})=1$ hence  $K_{Y_{2}}+D_{Y_2}$ is $\pi_2$-ample.
By explicit computation, $Y_2$ and  $D_{Y_2}$ are both smooth, thus
$\pi_2:(Y_2, D_{Y_{2}})\to (X_2, D_2)$ is the log canonical model.
Furthermore,  $\pi_2:D_{Y_2}\to D_2$ is the blow up of the origin,
hence it is not an isomorphism. (We note that $K_{X_2}+D_2$ is not
$\mathbb{Q}$-Cartier.)

Thus the isomorphism $\sigma: D_1\cong D_2$
gives a birational map $\sigma': D_{Y_1}\dashrightarrow D_{Y_2}$
that is not an isomorphism. Therefore
$(Y_1, D_{Y_1})$ and $(Y_2, D_{Y_2})$ can not be glued together.

\end{Ex}

\section*{Acknowledgement}
It is the authors' great pleasure to thank Professors Christopher Hacon, Masayuki Kawakita, J\'anos Koll\'ar  and Shigefumi  Mori for helpful e-mails and conversations. The joint work started when C.X.
visited RIMS, he would like to thank the warm hospitality and inspiring environment there.
Y.O is partially supported by the Grant-in-Aid for Scientific Research (KAKENHI No.\ 21-3748) and the Grant-in-Aid for JSPS fellows. C.X is partially supported by NSF research grant No.\ 0969495.

\end{document}